\documentclass[12pt]{amsart}

\usepackage{fullpage}
\usepackage{mathtools}
\usepackage{shuffle}
\mathtoolsset{showonlyrefs}
\usepackage{amscd,amsmath,amssymb,amsthm}

\usepackage{tikz}
\usepackage{tikz-cd}
\usetikzlibrary{matrix,arrows,decorations.pathmorphing}

\makeatletter
\@namedef{subjclassname@2020}{%
	\textup{2020} Mathematics Subject Classification}
\makeatother

\newtheorem{corollary}{Corollary}[section]
\newtheorem{lemma}[corollary]{Lemma}
\newtheorem{proposition}[corollary]{Proposition}

\newtheorem{theorem}[corollary]{Theorem}

\numberwithin{equation}{section}

\theoremstyle{definition}
\newtheorem{definition}[corollary]{Definition}
\newtheorem{example}[corollary]{Example}
\newtheorem{remark}[corollary]{Remark}

\usepackage{hyperref}

\hypersetup{%
	pdfborder = {0 0 0}
}

\begin{document}

\title[Iterated primitives of meromorphic quasimodular forms]{Iterated primitives of meromorphic quasimodular forms for $\operatorname{SL}_2(\mathbb Z)$}

\author{Nils Matthes}
\address{Mathematical Institute, University of Oxford, Andrew Wiles Building, Radcliffe Observatory Quarter, Woodstock Road, Oxford OX2 6GG, United Kingdom}
\email{nils.matthes@maths.ox.ac.uk}
\thanks{The author thanks Tiago J. Fonseca, Erik Panzer, Vicenţiu Paşol, and Wadim Zudilin for helpful discussions. Moreover, he thanks Claudia Alfes-Neumann, Tiago J. Fonseca, Jan-Willem van Ittersum, Masanobu Kaneko, Martin Raum, Wadim Zudilin, and the anonymous referee for valuable comments and suggestions on previous versions of the manuscript. Discussions with Johannes Broedel and Claude Duhr led to a significant simplification in the proof of Theorem 6.3, for which the author thanks both of them. Finally, special thanks go to Henrik Bachmann for organizing the Japan--Europe Number Theory Exchange seminar where the author first learned about the work of Paşol--Zudilin, which served as inspiration for the present paper. This project has received funding from the European Research Council (ERC) under the European Union’s Horizon 2020 research and innovation programme (grant agreement No. 724638).}

\address{Department of Mathematical Sciences, University of Copenhagen, Universitetsparken 5, 2100 Copenhagen Ø, DENMARK}


\date{\today}

\begin{abstract}
We introduce and study iterated primitives of meromorphic quasimodular forms for $\operatorname{SL}_2(\mathbb Z)$, generalizing work of Manin and Brown for holomorphic modular forms. We prove that the algebra of iterated primitives of meromorphic quasimodular forms is naturally isomorphic to a certain explicit shuffle algebra. We deduce from this an Ax--Lindemann--Weierstrass type algebraic independence criterion for primitives of meromorphic quasimodular forms which includes a recent result of Paşol--Zudilin as a special case. We also study spaces of meromorphic modular forms with restricted poles, generalizing results of Guerzhoy in the weakly holomorphic case.
\end{abstract}

\maketitle

\section{Introduction} \label{sec:1}

It is a well-known result in the theory of modular forms that the Eisenstein series
\[
E_2(\tau)=1-24\sum_{n=1}^{\infty}\frac{nq^n}{1-q^n},\qquad E_4(\tau)=1+240\sum_{n=1}^{\infty}\frac{n^3q^n}{1-q^n},\qquad E_6(\tau)=1-504\sum_{n=1}^{\infty}\frac{n^5q^n}{1-q^n},
\]
where $q:=e^{2\pi i\tau}$, are algebraically independent over the field $\mathbb C(q,\tau)$,  \cite{Mahler}. In other words, the rational function field $\mathbb C(E_2,E_4,E_6,q,\tau)$ has transcendence degree $5$. More recently, Paşol and Zudilin, \cite{PasolZudilin}, studied an extension of $\mathbb C(E_2,E_4,E_6,q,\tau)$ defined by adjoining primitives of the meromorphic modular forms
\[
\frac{\Delta(\tau)}{E_4(\tau)^2}, \qquad \frac{E_4(\tau)\Delta(\tau)}{E_6(\tau)^2}, \qquad \frac{E_6(\tau)\Delta(\tau)}{E_4(\tau)^3},
\]
where $\Delta(\tau)=(E_4(\tau)^3-E_6(\tau)^2)/1728$, and prove that these primitives are algebraically independent over $\mathbb C(E_2,E_4,E_6,q,\tau)$. Similar results for primitives of Eisenstein series were obtained previously by Kozlov, \cite{Kozlov}.

The purpose of this paper is to provide a general algebraic framework for such independence results in the case of general meromorphic quasimodular forms for the full modular group $\operatorname{SL}_2(\mathbb Z)$. The main idea is that every algebraic relation among primitives of meromorphic quasimodular forms can be rewritten as a linear relation among \emph{iterated primitives}. In the holomorphic case, these were introduced by Manin, \cite{Manin}, for cusp forms and by Brown, \cite{Brown:MMV}, in general; see also \cite{Matthes:AlgebraicStructure} for the quasimodular case. In this paper, we extend the definition of iterated primitives to the meromorphic case. Our approach is purely formal and only involves the Laurent expansion at the cusp $q=0$. Due to the possibility of higher order poles, the construction of iterated primitives is more involved than in the holomorphic case and requires a suitable renormalization scheme; see Section \ref{sec:4} for the details.

We next turn to the problem of describing the $K$-algebra $\mathcal{I}^{\mathcal{Q}\mathcal{M}}$ of iterated primitives of meromorphic quasimodular forms, where $K:=\mathbb C(E_2,E_4,E_6,q)$.\footnote{The reason for not adding $\tau$ to $K$ is that $\tau$ naturally occurs as a primitive, namely of the constant function $(2\pi i)^{-1} \in \mathcal{Q}\mathcal{M}$. In particular, $\tau \in \mathcal{I}^{\mathcal{Q}\mathcal{M}}$.} Let $\mathcal{Q}\mathcal{M}$ denote the graded $\mathbb C$-algebra of meromorphic quasimodular forms, and let $K\langle \mathcal{Q}\mathcal{M}\rangle$ denote the corresponding shuffle algebra. Its elements are $K$-linear combinations of words $[f_1|\ldots|f_n] \in \mathcal{Q}\mathcal{M}^{\otimes n}$, and mapping such a word to the corresponding iterated primitive defines a surjection of $K$-algebras
\begin{equation} \label{eqn:morphism}
I: K\langle \mathcal{Q}\mathcal{M}\rangle\rightarrow \mathcal{I}^{\mathcal{Q}\mathcal{M}}, \qquad [f_1|\ldots|f_n] \mapsto I(f_1,\ldots,f_n).
\end{equation}
Note that this map is far from being injective. In fact, for every $f \in \mathcal{Q}\mathcal{M}$, the element $[\delta(f)]-f \in K\langle \mathcal{Q}\mathcal{M}\rangle$, where $\delta=q\frac{d}{dq}$, lies in the kernel of $I$. On the other hand, it turns out that the existence of derivatives in $\mathcal{Q}\mathcal{M}$ is the only obstacle for $I$ to be injective. More precisely, let $\mathcal{C}$ be a $\mathbb C$-linear complement of the subspace $\delta(\mathcal{Q}\mathcal{M})\subset\mathcal{Q}\mathcal{M}$ of derivatives. The following theorem is our first main result.
\begin{theorem}[see Theorem \ref{thm:main} below] \label{thm:intro1}
The morphism \eqref{eqn:morphism} restricts to an isomorphism of $K$-algebras:
\[
K\langle \mathcal{C}\rangle\cong \mathcal{I}^{\mathcal{Q}\mathcal{M}}.
\]
\end{theorem}
This theorem is a generalization of the main result of \cite{Matthes:AlgebraicStructure} to the meromorphic case. As in \textit{loc.cit.}, the hard part of the proof is injectivity for which we employ a general linear independence criterion for solutions of first order differential equations, \cite[Theorem 2.1]{DDMS}. In our situation, this criterion amounts to verifying that, if the primitive of a meromorphic quasimodular form is contained in $K$, then that primitive must already lie in the subring $\mathcal{Q}\mathcal{M}$; see Lemma \ref{lem:main} below.

Theorem \ref{thm:intro1} can be viewed as a quite general independence result for iterated primitives of meromorphic quasimodular forms. In fact, the Cartier--Milnor--Moore theorem implies that $K\langle \mathcal{C}\rangle$ is a polynomial algebra for which Radford, \cite{Radford}, has given explicit generators in terms of certain Lyndon words. Combining his result with Theorem \ref{thm:intro1} leads to the following Ax--Lindemann--Weierstrass type criterion for algebraic independence.
\begin{theorem}[see Theorem \ref{thm:criterion} below for a more general result] \label{thm:intro2}
	Given $f_1,\ldots,f_n \in \mathcal{Q}\mathcal{M}$, their primitives $I(f_1),\ldots,I(f_n)$ are algebraically independent over $K$ if and only if the classes of $f_1,\ldots,f_n$ in the quotient $\mathcal{Q}\mathcal{M}/\delta(\mathcal{Q}\mathcal{M})$ are $\mathbb C$-linearly independent.
\end{theorem}
At this point, both Theorems \ref{thm:intro1} and \ref{thm:intro2} are of somewhat limited use since we do not yet have any nontrivial information about the space $\mathcal{C} \cong \mathcal{Q}\mathcal{M}/\delta(\mathcal{Q}\mathcal{M})$. To address this, let $\mathcal{M}_k\subset \mathcal{Q}\mathcal{M}$ denote the subspace of meromorphic modular forms of weight $k$ and, for $k\geq 2$, let $\widetilde{\mathcal{M}}_k\subset \mathcal{M}_k$ denote the subspace of those meromorphic modular forms whose pole order at the cusp $\infty$ is bounded (from above) by $\dim S_k$, and whose pole order at any other point is bounded by $k-1$. A similar space was studied by Guerzhoy, \cite{Guerzhoy}, in the more restrictive context of weakly holomorphic modular forms, where only poles at $\infty$ are allowed.
\begin{theorem}[see Theorem \ref{thm:complement} below] \label{thm:intro3}
	\[
	\mathcal{Q}\mathcal{M}_k=
	\begin{cases}
		\delta(\mathcal{Q}\mathcal{M}_{k-2})\oplus \mathcal{M}_k, &\mbox{if } k\leq 1,\\
		\delta(\mathcal{Q}\mathcal{M}_{k-2})\oplus \mathcal{M}_{2-k}\cdot E_2^{k-1} \oplus \widetilde{\mathcal{M}}_k, & \mbox{if } k\geq 2
	\end{cases}
	\]
\end{theorem}
The analogous statement for holomorphic quasimodular forms can be deduced from \cite[Theorem 4.2]{Royer}. Also, the conjunction of Theorems \ref{thm:intro2} and \ref{thm:intro3} implies the independence result of Paşol--Zudilin mentioned above; see Corollary \ref{cor:PasolZudilin} below. Similarly, applying Theorem \ref{thm:intro2} (or rather its generalization, Theorem \ref{thm:complement}) to the family of holomorphic Eisenstein series $E_{2k}$, including $E_0:=-1$, gives a new proof of Kozlov's result, \cite[Theorem 1]{Kozlov}.

The key step in the proof of Theorem \ref{thm:intro3} is to establish that the subspace $\widetilde{\mathcal{M}}_k\subset \mathcal{M}_k$ splits the \emph{Bol space}, \cite{Bol},
\[
\mathcal{B}_k:=\mathcal{M}_k/\delta^{k-1}(\mathcal{M}_{2-k}).
\]
That is, the natural map $\varphi_k: \widetilde{\mathcal{M}}_k\rightarrow \mathcal{B}_k$ is an isomorphism. This generalizes a result of Guerzhoy in the weakly holomorphic case; see Theorem \ref{thm:splitting} below. We also obtain dimension formulas for the restricted Bol space $B_k(\ast S) \subset \mathcal{B}_k$, where possible poles are confined to a finite subset $S$ of the level one modular curve $X$. The precise result is as follows.
\begin{theorem}[see Theorem \ref{thm:dimension} below] \label{thm:intro4}
	Let $S\subset X$ be a finite subset with $\infty \in S$, and let $k\geq 2$ be an even integer. Then
	\[
	\dim B_k(\ast S)=\dim M_k+\dim S_k+\sum_{P\in S'}w_P(k),
	\]
	where $S':=S\setminus \{\infty\}$ and
	\[
	w_P(k):=\begin{cases}
		2\left\lfloor \frac{k-2}{4}\right\rfloor+1, & \mbox{if }P=[i],\\
		2\left\lfloor \frac{k-2}{6}\right\rfloor+1, &  \mbox{if }P=[e^{2\pi i/3}],\\
		k-1, &\mbox{else.}
	\end{cases}
	\]
\end{theorem}
The case $S=\{\infty\}$ is essentially \cite[Theorem 1]{Guerzhoy}. We would like to point out that the space $B_k(\ast \{\infty\})$ admits an interpretation in terms of algebraic de Rham cohomology (with coefficients) of the moduli stack of elliptic curves, \cite{BrownHain}. It would be interesting to find a similar interpretation of the spaces $B_k(\ast S)$, but this is beyond the scope of this paper.\footnote{After this paper was submitted for publication, the author has been informed by Tiago J. Fonseca that it is possible to deduce Theorem \ref{thm:intro4} from the Eichler--Shimura theorem together with general properties of algebraic de Rham cohomology with coefficients. The details will appear in forthcoming work of Francis Brown and Fonseca.}

\subsection*{Contents}
Sections \ref{sec:2} and \ref{sec:3} contain reminders on meromorphic modular and quasimodular forms respectively. In Section \ref{sec:4}, we construct iterated primitives of meromorphic quasimodular forms. We then prove Theorems \ref{thm:intro1} and \ref{thm:intro2} in Section \ref{sec:5}, and finally prove Theorems \ref{thm:intro3} and \ref{thm:intro4} in Section \ref{sec:6}.

\section{Reminder on meromorphic modular forms} \label{sec:2}

Let $\mathfrak{H}=\{\tau\in \mathbb C \, : \, \operatorname{Im}(\tau)>0\}$ be the upper half-plane. Recall that a meromorphic function $f$ on $\mathfrak{H}$ is said to be \emph{weakly modular of weight $k\in \mathbb Z$}, if $f[\gamma]_k=f$ for all $\gamma \in \operatorname{SL}_2(\mathbb Z)$, where
\[
f[\gamma]_k(\tau):=(c\tau+d)^{-k}f\left(\frac{a\tau+b}{c\tau+d}\right), \qquad \mbox{for }\gamma=\begin{pmatrix}a&b\\c&d\end{pmatrix}.
\]
Each weakly modular function has a Fourier expansion $f(\tau)=\sum_{n=-\infty}^\infty a_nq^n$, where $q=e^{2\pi i\tau}$. We say that $f$ is \emph{meromorphic at $\infty$}, if there exists $N\in \mathbb Z$ such that $a_n=0$ for all $n<N$. The smallest such $N$ is denoted by $v_\infty(f)$ with the convention that $v_\infty(f)=\infty$ if no such $N$ exists. This happens, of course, if and only if $f$ is the zero function. If $v_\infty(f)\geq 0$, then we say that $f$ is \emph{holomorphic at $\infty$}.
\begin{definition}
	A \emph{meromorphic modular form of weight $k \in \mathbb Z$} is a meromorphic function $f$ on $\mathfrak{H}$ which is both weakly modular of weight $k$ and meromorphic at $\infty$.
	
	We let $\mathcal{M}_k$ denote the $\mathbb C$-vector space of meromorphic modular forms of weight $k$, and let $\mathcal{M}:=\bigoplus_{k\in \mathbb Z}\mathcal{M}_k$ denote the corresponding graded $\mathbb C$-algebra.
\end{definition}
Now let $\mathfrak{H}^*:=\mathfrak{H}\cup \mathbf P^1(\mathbb{Q})$ be the extended upper half-plane. The group $\operatorname{SL}_2(\mathbb Z)$ acts on $\mathfrak{H}^*$ in the usual way via M\"obius transformations and the quotient $X:=\operatorname{SL}_2(\mathbb Z)\backslash\mathfrak{H}^*$ is known as the \emph{modular curve of level one}. It has a natural structure of compact Riemann surface and exactly one cusp $\infty \in X$.

Given a meromorphic modular form $f$ and a point $p\in \mathfrak{H}$, the vanishing order $v_p(f)$ of $f$ at $p$ depends only on the class $[p] \in X$, and we shall write $v_{[p]}(f)$ instead of $v_p(f)$. Therefore, if $f\neq 0$, we can associate to $f$ its \emph{divisor}
\[
\operatorname{div}(f)=\sum_{P\in X}\frac{v_P(f)}{h_P}\cdot (P), \qquad \mbox{where} \qquad h_P:=
\begin{cases}
	2, & \mbox{if }P=[i],\\
	3, & \mbox{if }P=[e^{2\pi i/3}],\\
	1, & \mbox{else.}
\end{cases}
\]
Note that the valence formula implies that $\deg \operatorname{div}(f)=k/12$. Conversely, let $D=\sum_{P\in X}\frac{n_P}{h_P}\cdot (P)$ be a divisor on $X$.
\begin{proposition} \label{prop:isomorphism}
	There exists a meromorphic modular form $g_D \in \mathcal{M}_{12\deg D}$ such that $\operatorname{div}(g_D)=D$.
\end{proposition}
\begin{proof}
	The function
	\[
	u_P(\tau)=
	\begin{cases}
		E_4(\tau), & \mbox{if }P=[e^{2\pi i/3}],\\
		E_6(\tau), & \mbox{if }P=[i]\\
		\Delta(\tau), & \mbox{if }P=\infty,\\
		(j(\tau)-j(P))\Delta(\tau), & \mbox{else,}
	\end{cases}
	\]
	where $j(\tau)=E_4(\tau)^3/\Delta(\tau)$, is a holomorphic modular form of weight $12/h_P$ such that $\operatorname{div}(u_P)=\frac{1}{h_P}\cdot (P)$. Therefore, $g_D=\prod_{P\in X}u_P^{n_P} \in \mathcal{M}_{12\deg D}$ has the requisite property.
\end{proof}
\begin{remark} \label{rmk:isomorphism}
	Proposition \ref{prop:isomorphism} has the following consequences. First of all, let
	\[
	M_k(D)=\{f\in \mathcal{M}_k\, : \, v_P(f)\geq -n_P\}.
	\]
	If $D=0$, then $M_k(D)=M_k$, the space of holomorphic modular forms of weight $k$. Then Proposition \ref{prop:isomorphism} implies that $M_k(D)=g_D^{-1}\cdot M_{k+12\deg D}$, so that in particular
	\begin{equation} \label{eqn:dimension}
		\dim M_k(D)=\dim M_{k+12\deg D}.
	\end{equation}
	Secondly, it follows from Proposition \ref{prop:isomorphism} that every $f\in \mathcal{M}_k$ can be written as a quotient $f=g/h$, where $g,h$ are holomorphic modular forms, with $h\neq 0$, of weights $k_g$ and $k_h$, respectively, such that $k=k_g-k_h$. Since every holomorphic modular form can be written as a polynomial in $E_4$ and $E_6$, it follows that $\mathcal{M}\subset \mathbb C(E_4,E_6)$ is the subspace spanned by homogeneous rational functions in $E_4$ and $E_6$.
\end{remark}
The following lemma gives a construction of meromorphic modular forms with certain prescribed poles that will be needed later.
\begin{lemma} \label{lem:construction}
	Let $k\geq 2$ be an even integer, $f \in \mathcal{M}_k$ be a meromorphic modular form, and $P\in X\setminus \{\infty\}$.
	\begin{itemize}
		\item[(i)] If $v_P(f)\leq -k$, then there exists $g\in \mathcal{M}_{2-k}$ such that $v_P(g)=v_P(f)+k-1$ and $v_Q(g) \geq 0$, for all $Q\in X\setminus \{P,\infty\}$.
		\item[(ii)]
		If $v_\infty(f)\leq -\dim S_k-1$, then there exists $g\in \mathcal{M}_{2-k}$ such that $v_\infty(g)=v_\infty(f)$ and $v_P(g)\geq 0$, for all $P\in X\setminus \{\infty\}$.
	\end{itemize}
\end{lemma}
\begin{proof}
	\begin{itemize}
		\item[(i)]
		\textit{Case 1: $P \notin\{[i],[e^{2\pi i/3}]\}$.}
		There exist $a,b,n_\infty \in \mathbb Z$ with $a,b\geq 0$ such that
		\[
		12n_\infty+4a+6b=2-k-12(v_P(f)+k-1),
		\]
		and $g=u_P^{v_P(f)+k-1}\Delta^{n_\infty}E_4^aE_6^b$ satisfies the conditions in (i).
		\bigskip
		
		\noindent
		\textit{Case 2: $P=[i]$.}
		We claim that there exist $a,n_\infty \in \mathbb Z$ with $a\geq 0$ such that
		\begin{equation} \label{eqn:linearequation}
			12n_\infty+4a=2-k-6(v_P(f)+k-1).
		\end{equation}
		Indeed, the right hand side of \eqref{eqn:linearequation} is divisible by $4$, as a consequence of the valence formula. Hence $a,n_\infty$ as above exist and $g=\Delta^{n_\infty}E_4^aE_6^{v_P(f)+k-1}$ satisfies the conditions in (i).
		\bigskip
		
		\noindent
		\textit{Case 3: $P=[e^{2\pi i/3}]$.}
		Similarly to before, one shows that there exist $b,n_\infty \in \mathbb Z$ with $b\geq 0$ such that
		\begin{equation} \label{eqn:linearequation2}
			12n_\infty+6b=2-k-4(v_P(f)+k-1).
		\end{equation}
		Then $g=\Delta^{n_\infty}E_4^{v_P(f)+k-1}E_6^b$ satisfies the conditions in (i).
		\item[(ii)]
		Similarly to the proof of (i), we claim that there exist $a,b\geq 0$ such that
		\begin{equation} \label{eqn:linearequation3}
			4a+6b=2-k-12v_\infty(f).
		\end{equation}
		Indeed, using our assumption on $v_\infty(f)$ and the well-known dimension formula for $\dim S_k$, one verifies that the right hand side of \eqref{eqn:linearequation3} is a non-negative integer $\neq 2$, hence $a,b$ as above exist. Then $g=\Delta^{v_\infty(f)}E_4^aE_6^b$ satisfies the conditions of (ii).
	\end{itemize}
\end{proof}

\section{Meromorphic quasimodular forms} \label{sec:3}

The notion of holomorphic quasimodular form was introduced in \cite{KanekoZagier}; see also \cite{Royer} for an expository account. The generalization to the meromorphic case is straightforward.

Let $k,p$ be integers with $p\geq 0$. We say that a meromorphic function $f$ on $\mathfrak{H}$ is \emph{weakly quasimodular\footnote{This terminology is not standard.} of weight $k$ and depth $p$} if there exist meromorphic functions $f_0,\ldots,f_p$ on $\mathfrak{H}$ such that
\begin{equation} \label{eqn:weaklyquasimodular}
	f[\gamma]_k(\tau)=\sum_{r=0}^p f_r(\tau)\cdot X_\gamma(\tau)^r, \qquad\mbox{where } X_\gamma(\tau)=\frac{12}{2\pi i}\frac{c}{c\tau+d},
\end{equation}
for all $\gamma \in \operatorname{SL}_2(\mathbb Z)$. The functions $f_0,\ldots,f_p$ are then uniquely determined by $f$ and are referred to as the \emph{coefficient functions} of $f$. Note that $f_0=f$ and, more generally, that each $f_r$ is itself weakly quasimodular, of weight $k-2r$ and depth $p-r$. More precisely,
\begin{equation} \label{eqn:weaklyquasimodularcoefficient}
	f_r[\gamma]_{k-2r}=\sum_{s=r}^p\binom{s}{r}f_s\cdot X_\gamma^{s-r},
\end{equation}
for all $\gamma \in \operatorname{SL}_2(\mathbb Z)$. In particular, $f_r(\tau+1)=f_r(\tau)$, so that each $f_r$ has a Fourier expansion. 
\begin{definition} \label{dfn:quasimodular}
	A \emph{meromorphic quasimodular form of weight $k\in \mathbb Z$ and depth $p\geq 0$} is a meromorphic function $f$ on $\mathfrak{H}$ which satisfies the following two conditions:
	\begin{itemize}
		\item[(i)]
		$f$ is weakly quasimodular of weight $k$ and depth $p$;
		\item[(ii)]
		The functions $f_0,\ldots,f_p$ are meromorphic at $\infty$
	\end{itemize}
	We let $\mathcal{Q}\mathcal{M}_k^{\leq p}$ denote the $\mathbb C$-vector space of meromorphic quasimodular forms of weight $k$ and depth $p$. We also define
	\[
	\mathcal{Q}\mathcal{M}_k=\bigcup_{p\geq 0}\mathcal{Q}\mathcal{M}_k^{\leq p}, \qquad \mbox{respectively} \qquad \mathcal{Q}\mathcal{M}=\bigoplus_{k\in \mathbb Z}\mathcal{Q}\mathcal{M}_k,
	\]
	to be the $\mathbb C$-vector space of meromorphic quasimodular forms of weight $k$, and the graded $\mathbb C$-algebra of all meromorphic quasimodular forms, respectively.
\end{definition}
Now let $f \in \mathcal{Q}\mathcal{M}_k^{\leq p}$ with coefficient functions $f_0,\ldots,f_p$, and let $\delta=q\frac{d}{dq}=\frac{1}{2\pi i}\frac{d}{d\tau}$ denote the $q$-derivative. Applying $\delta$ to both sides of \eqref{eqn:weaklyquasimodular}, we deduce that
\begin{equation} \label{eqn:differentiation}
	\delta(f)[\gamma]_{k+2}=\sum_{r=0}^{p+1}\left(\delta(f_r)+\frac{k-r+1}{12}f_{r-1}\right) X_\gamma^r, \qquad f_{-1}=f_{p+1}:=0,
\end{equation}
for all $\gamma \in \operatorname{SL}_2(\mathbb Z)$. In particular, there is a well-defined map
\[
\delta: \mathcal{Q}\mathcal{M}_k^{\leq p}\rightarrow \mathcal{Q}\mathcal{M}_{k+2}^{\leq p+1},
\]
for all $k,p \in \mathbb Z$ with $p\geq 0$.
\begin{example}
	The holomorphic Eisenstein series $E_2(\tau)$ is a holomorphic quasimodular form of weight $2$ and depth $1$. It satisfies 
	\begin{equation} \label{eqn:E2quasimodular}
		E_2[\gamma]_2=E_2+X_\gamma,
	\end{equation}
	for all $\gamma \in \operatorname{SL}_2(\mathbb Z)$. Every meromorphic quasimodular form can be written uniquely as a polynomial in $E_2$ whose coefficients are meromorphic modular forms. More precisely, we have
	\begin{equation} \label{eqn:polynomialinE2}
		\mathcal{Q}\mathcal{M}_k^{\leq p}=\bigoplus_{0\leq r\leq p}\mathcal{M}_{k-2r}\cdot E_2^r, \qquad \mbox{and} \qquad \mathcal{Q}\mathcal{M}=\mathcal{M}[E_2].
	\end{equation}
	In particular, the depth of $f$ equals its degree in $E_2$.
\end{example}

\section{Iterated primitives of formal Laurent series} \label{sec:4}

The goal of this section is to define iterated primitives of meromorphic quasimodular forms. More generally, we define iterated primitives of formal Laurent series in one variable $q$, which can be specialized to meromorphic quasimodular forms using their Fourier expansion. The case of formal power series is technically simpler and will be dealt with first. The general case requires renormalization techniques which originated in quantum field theory, \cite{ConnesKreimer}; see also \cite{Manchon} for an expository account.

\subsection{The case of formal power series}

Let $\mathcal{Q}\mathcal{M}^{\infty}\subset \mathcal{Q}\mathcal{M}$ denote the $\mathbb C$-subalgebra of meromorphic quasimodular forms which are holomorphic at $\infty$. Mapping a meromorphic quasimodular form to its Fourier expansion induces an embedding of $\mathbb C$-algebras
\[
\mathcal{Q}\mathcal{M}^{\infty} \hookrightarrow \mathbb C[\![q]\!].
\]
Moreover, the $q$-derivative $\delta=q\frac{d}{dq}$ extends to $\mathbb C[\![q]\!]$ in the natural way and every $f=\sum_{n\geq 0}a_nq^n \in \mathbb C[\![q]\!]$ with $a_0=0$ has a $\delta$-primitive in $\mathbb C[\![q]\!]$, namely $\sum_{n\neq 0}n^{-1}a_nq^n$.

In order to deal with the case where $a_0\neq 0$, we introduce a new variable $t$ such that $\delta(t)=1$, and define a $\mathbb C$-algebra by
$\mathcal{A}^{\infty}:=(\mathbb C[\![q]\!])[t]$. One can show that every element $f\in \mathcal{A}^{\infty}$ has a primitive $F\in \mathcal{A}^{\infty}$, which is unique up to adding an element of $\ker(\delta)=\mathbb C$.
\begin{definition} \label{dfn:primitive}
	Define a formal integration map
	\[
	I: \mathcal{A}^{\infty}\rightarrow \mathcal{A}^{\infty}, \qquad f\mapsto F-\operatorname{ev}_0(F),
	\]
	where $F=\sum_{i=0}^nF_i(q)t^i$ is any primitive of $f$, and $\operatorname{ev}_0:\mathcal{A}^{\infty}\rightarrow \mathbb C$ is defined by $\operatorname{ev}_0(F)=F_0(0)$.
\end{definition}
Equivalently, $I(f)$ is defined as the unique primitive of $f$ whose constant term (as a series in $t$ and $q$) vanishes.
\begin{definition} \label{dfn:iteratedprimitive}
	Given $f_1,\ldots,f_n \in \mathcal{A}^{\infty}$, define their \emph{iterated primitive} recursively by
	\[
	I(f_1,\ldots,f_n)=
	\begin{cases}
		I(f_1\cdot I(f_2,\ldots,f_n)), & \mbox{if }n\geq 1, \\
		1, & \mbox{if }n=0.
	\end{cases}
	\]
	The \emph{length} of $I(f_1,\ldots,f_n)$ is defined to be the integer $n$.
\end{definition}
The next proposition states some standard properties of iterated primitives (cf. \cite{Chen}). Since our formal setup is slightly different, we give a proof for the reader's convenience.
\begin{proposition} \label{prop:iteratedprimitiveprops}
	The following statements are true.
	\begin{itemize}
		\item[(i)] For all $f_1,\ldots,f_n \in \mathcal{A}^{\infty}$, we have
		\[
		\delta(I(f_1,\ldots,f_n))=f_1\cdot I(f_2,\ldots,f_n).
		\]
		\item[(ii)] For all integers $m,n\geq 0$, and all $f_1,\ldots,f_{m+n} \in \mathcal{A}^{\infty}$, the \emph{shuffle product formula} holds:
		\begin{equation} \label{eqn:shuffleproduct}
			I(f_1,\ldots,f_m)I(f_{m+1},\ldots,f_{m+n})=\sum_{\sigma\in \Sigma_{m,n}}I(f_{\sigma^{-1}(1)},\ldots,f_{\sigma^{-1}(m+n)}),
		\end{equation}
		where $\Sigma_{m,n}$ denotes the set of all permutations $\sigma: \{1,\ldots,m+n\}\rightarrow \{1,\ldots,m+n\}$ such that $\sigma(1)<\ldots<\sigma(m)$ and $\sigma(m+1)<\ldots<\sigma(m+n)$.
	\end{itemize}
\end{proposition}
\begin{proof}
	Statement (i) is immediate. To prove (ii), we use induction on $m+n$. For $m+n\leq 1$, the statement is trivial. In general, applying $\delta$ to the left hand side of \eqref{eqn:shuffleproduct} gives
	\begin{equation} \label{eqn:1}
		f_1\cdot I(f_2,\ldots,f_m)I(f_{m+1},\ldots,f_{m+n})+f_{m+1}\cdot I(f_1,\ldots,f_m)I(f_{m+2},\ldots,f_{m+n}),
	\end{equation}
	whereas applying $\delta$ to the right hand side yields 
	\begin{equation} \label{eqn:2}
		\sum_{\sigma \in \Sigma_{m,n}}f_{\sigma^{-1}(1)}\cdot I(f_{\sigma^{-1}(2)},\ldots,f_{\sigma^{-1}(m+n)}).
	\end{equation}
	Identifying $\Sigma_{m-1,n}$ with the permutations in $\Sigma_{m,n}$ which satisfy $\sigma(1)=1$, and likewise $\Sigma_{m,n-1}$ with the ones which satisfy $\sigma(m+1)=1$, we deduce that $\Sigma_{m,n}$ is the disjoint union of $\Sigma_{m-1,n}$ and $\Sigma_{m,n-1}$. Together with the induction hypothesis, it follows that \eqref{eqn:1} and \eqref{eqn:2} are equal, hence that \eqref{eqn:shuffleproduct} holds up to a constant. Moreover, as $\operatorname{ev}_0: \mathcal{A}^{\infty}\rightarrow \mathbb C$ is an algebra homomorphism, one verifies that both sides of \eqref{eqn:shuffleproduct} vanish upon applying $\operatorname{ev}_0$. Hence, the constant in question is zero, and we conclude.
\end{proof}
\begin{remark}
	In the case where the $f_1,\ldots,f_n \in \mathcal{A}^{\infty}$ are Fourier expansions of holomorphic modular forms, Definition \ref{dfn:iteratedprimitive} specializes to the iterated Eichler--Shimura integrals introduced by Manin, \cite{Manin}, and Brown, \cite{Brown:MMV}. More precisely, with notation and conventions as in \cite{Brown:MMV}, we have
	\[
	\begin{aligned}
		I(f_1,\ldots,f_n)=\frac{1}{(-2\pi i)^n}\int_\tau^{\overrightarrow{1}_\infty}f_1(\tau_1)d\tau_1\ldots f_n(\tau_n)d\tau_n.
	\end{aligned}
	\]
	This follows from \cite[Proposition 4.7.i)]{Brown:MMV} and the fact that the expansion in $q$ and $\tau$ of the right hand side has no constant term, by construction.
\end{remark}

\subsection{The general case via renormalization}

Now let $\mathbb C(\!(q)\!):=\operatorname{Frac}(\mathbb C[\![q]\!])$ be the $\mathbb C$-algebra of finite-tailed, formal Laurent series and let $\mathcal{A}:=\mathbb C (\!(q)\!)[t]$. Similarly to before, mapping a meromorphic quasimodular form to its Fourier expansion defines an embedding $\mathcal{Q}\mathcal{M}\hookrightarrow \mathcal{A}$. Also, the integration map $I$ extends to $\mathcal{A}$ by defining $I(f)=F-a_{0,0}$, where $F=\sum_{i=0}^nF_i(q)t^i$, with $F_i(q)=\sum_{n>\!\!\!>-\infty}a_{i,n}q^n$ is any primitive of $f$. However, the extension of Definition \ref{dfn:iteratedprimitive} to $f_1,\ldots,f_n \in \mathcal{A}$ does in general not satisfy the shuffle product formula, the problem being that the map $\mathcal{A}\rightarrow \mathbb C$ sending an element of $\mathcal{A}$ to its constant term (viewed as a series in $q$ and $t$) is not an algebra homomorphism.

In order to resolve this, the first step is to introduce suitably perturbed iterated primitives $I_\varepsilon(f_1,\ldots,f_n)$. Let $t_\varepsilon,\varepsilon$ be free variables, and define $\mathbb C(\!(q,\varepsilon)\!)$ to be the $\mathbb C$-algebra of formal Laurent series
\begin{equation}\label{eqn:formalseries}
	f=\sum_{m,n>\!\!\!>-\infty}a_{m,n}q^m\varepsilon^n, \qquad a_{m,n} \in \mathbb C
\end{equation}
and let $\mathcal{A}_\varepsilon:=\mathbb C(\!(q,\varepsilon)\!)[t,t_\varepsilon]$. The $q$-derivative extends trivially to $\mathcal{A}_\varepsilon$ by $\delta(t_\varepsilon)=\delta(\varepsilon)=0$ and every $f\in \mathcal{A}_\varepsilon$ has a $\delta$-primitive, unique up to adding an element of $\ker(\delta)=\mathbb C(\!(\varepsilon)\!)[t_\varepsilon]$. Therefore, similarly to Definition \ref{dfn:primitive}, we can define a formal integration map by
\[
I_\varepsilon: \mathcal{A}_\varepsilon\rightarrow \mathcal{A}_\varepsilon, \qquad f\mapsto F-\operatorname{ev}_\varepsilon(F),
\]
where $F$ is any primitive of $f$ and $\operatorname{ev}_\varepsilon: \mathcal{A}_\varepsilon \rightarrow \mathbb C(\!(\varepsilon)\!)[t_\varepsilon]$ is defined by substituting $(t,q)\mapsto (t_\varepsilon,\varepsilon)$. Now, given elements $f_1,\ldots,f_n\in \mathcal{A}$, we define $I_\varepsilon(f_1,\ldots,f_n) \in \mathcal{A}_\varepsilon$ recursively by
\[
I_\varepsilon(f_1,\ldots,f_n)=
\begin{cases}
	I_\varepsilon(f_1\cdot I_\varepsilon(f_2,\ldots,f_n)), & \mbox{if }n\geq 1,\\
	1, & \mbox{if }n=0.
\end{cases}
\]
Using that $\operatorname{ev}_\varepsilon$ is an algebra homomorphism, it is straightforward to verify that the analogue of Proposition \ref{prop:iteratedprimitiveprops} holds for the $I_\varepsilon(f_1,\ldots,f_n)$.

We next modify the $I_\varepsilon(f_1,\ldots,f_n)$ by adding suitable counterterms to cancel the divergences as $\varepsilon \to 0$. Let $\mathcal{A}_+ \subset \mathcal{A}_\varepsilon$ (respectively, $\mathcal{A}_-\subset \mathcal{A}_\varepsilon$) denote the subset of those elements $f(q,\varepsilon,t,t_\varepsilon)$ which are analytic at $\varepsilon=0$ (respectively, such that $f(q,\varepsilon^{-1},t,t_\varepsilon)$ is analytic at $\varepsilon=0$). There is the direct sum decomposition
\[
\mathcal{A}_\varepsilon=\mathcal{A}_+\oplus \mathcal{A}_-,
\]
and we let $\pi_+: \mathcal{A}_\varepsilon\rightarrow \mathcal{A}_+$, and $\pi_-: \mathcal{A}_\varepsilon\rightarrow \mathcal{A}_-$ denote the canonical projections.
\begin{definition}
	Given $f_1,\ldots,f_n\in \mathcal{A}$, define $I_+(f_1,\ldots,f_n) \in \mathcal{A}_+$ and $I_-(f_1,\ldots,f_n)\in \mathcal{A}_-$ recursively by
	\[
	\begin{aligned}
		I_+(f_1,\ldots,f_n)&=\pi_+\left(I_\varepsilon(f_1,\ldots,f_n)+\sum_{i=1}^{n-1}I_\varepsilon(f_1,\ldots,f_i)I_-(f_{i+1},\ldots,f_n)\right),\\
		I_-(f_1,\ldots,f_n)&=-\pi_-\left(I_\varepsilon(f_1,\ldots,f_n)+\sum_{i=1}^{n-1}I_\varepsilon(f_1,\ldots,f_i)I_-(f_{i+1},\ldots,f_n)\right).
	\end{aligned}
	\]
\end{definition}
The next proposition, analogous to Proposition \ref{prop:iteratedprimitiveprops}, is crucial.
\begin{proposition} \label{prop:iteratedprimitivepropsII}
	The following statements are true.
	\begin{itemize}
		\item[(i)] For all $f_1,\ldots,f_n \in \mathcal{A}$, we have $\delta(I_-(f_1,\ldots,f_n))=0$ and
		\[
		\delta(I_+(f_1,\ldots,f_n))=f_1\cdot I_+(f_2,\ldots,f_n)
		\]
		\item[(ii)] For all integers $m,n\geq 0$ and all $f_1,\ldots,f_{m+n} \in \mathcal{A}$, we have
		\begin{equation} \label{eqn:shuffleproductII}
			I_\pm(f_1,\ldots,f_m)I_\pm(f_{m+1},\ldots,f_{m+n})=\sum_{\sigma\in \Sigma_{m,n}}I_\pm(f_{\sigma^{-1}(1)},\ldots,f_{\sigma^{-1}(m+n)}).
		\end{equation}
	\end{itemize}
\end{proposition}
\begin{proof}
	For statement (i), it is enough to prove that $\delta(I_-(f_1,\ldots,f_n))=0$; the other part follows from this using Proposition \ref{prop:iteratedprimitiveprops}.(i) and the Leibniz rule. If $n=1$, then
	\[
	\delta(I_-(f))=-f_1\cdot \pi_-(1)=0,
	\]
	since $\pi_-$ is $\mathcal{A}$-linear and commutes with $\delta$. Similarly, by induction on $n$, we have
	\[
	\begin{aligned}
		\delta(I_-(f_1,\ldots,f_n))&=-f_1\cdot \pi_-\Bigg(I_\varepsilon(f_2,\ldots,f_n)+\sum_{i=1}^{n-1}I_\varepsilon(f_2,\ldots,f_i)I_-(f_{i+1},\ldots,f_n)\Bigg)\\
		&=f_1\cdot (I_-(f_2,\ldots,f_n)-\pi_-(I_-(f_2,\ldots,f_n)))\\
		&=0
	\end{aligned}
	\]
	since $\pi_-$ is the identity on $\mathcal{A}_-$. This proves (i). Statement (ii) is a special case of a general result in renormalization; see \cite[Theorem II.5.1.(3)]{Manchon}.
\end{proof}
We are now in a position to define iterated primitives of general Laurent series.
\begin{definition} \label{dfn:iteratedprimitiveII}
	Given $f_1,\ldots,f_n \in \mathcal{A}$, define their \emph{iterated primitive} as
	\[
	I(f_1,\ldots,f_n)=\pi(I_+(f_1,\ldots,f_n)),
	\]
	where $\pi: \mathcal{A}_+\rightarrow \mathcal{A}$ is the canonical projection onto the constant term in $\varepsilon$ and $t_\varepsilon$.
\end{definition}
\begin{remark}
	Since $I_-(f_1,\ldots,f_n)=0$, if $f_1,\ldots,f_n\in \mathcal{A}^{\infty}$, Definition \ref{dfn:iteratedprimitiveII} is compatible with Definition \ref{dfn:iteratedprimitive}.
\end{remark}

\subsection{A worked example}

We illustrate Definition \ref{dfn:iteratedprimitiveII} in a simple example. Let $f_1=q^{-1}$ and $f_2=q$. Since $\delta(q^n)=nq^n$, we have
\[
I_\varepsilon(f_1)=-q^{-1}+\varepsilon^{-1}, \qquad \mbox{and} \qquad I_\varepsilon(f_2)=q-\varepsilon,
\]
so that
\[
I_+(f_1)=-q^{-1}, \qquad I_-(f_1)=-\varepsilon^{-1}, \qquad \mbox{and} \qquad I_+(f_2)=q-\varepsilon, \qquad I_-(f_2)=0.
\]
It follows that
\[
I(f_1)=I_+(f_1)=-q^{-1}, \qquad \mbox{and} \qquad I(f_2)=q.
\]
On the other hand, since $\delta(t)=1$, we have $I_\varepsilon(1)=t-t_\varepsilon$. Using this, we compute
\[
I_\varepsilon(f_1,f_2)=t-t_\varepsilon+q^{-1}\varepsilon-1, \qquad I_\varepsilon(f_2,f_1)=-(t-t_\varepsilon)+q\varepsilon^{-1}-1,
\]
as well as
\[
I_+(f_1,f_2)=t-t_\varepsilon+q^{-1}\varepsilon-1, \qquad I_+(f_2,f_1)=-t+t_\varepsilon.
\]
Finally, applying the projection $\pi: \mathcal{A}_+\rightarrow \mathcal{A}$ gives
\[
I(f_1,f_2)=t-1, \qquad I(f_2,f_1)=-t.
\]
We also observe that the shuffle product formula $I(f_1)I(f_2)=I(f_1,f_2)+I(f_2,f_1)$ holds and that
\begin{equation} \label{eqn:integrationbypartsexample}
	I(\delta(f_2),f_1)=f_2\cdot I(f_1)-I(f_1\cdot f_2)+1.
\end{equation}
\subsection{Some basic properties of renormalized iterated integrals}

Proposition \ref{prop:iteratedprimitiveprops} generalizes to the meromorphic case.
\begin{proposition} \label{prop:iteratedprimitivepropsIII}
	The following statements are true.
	\begin{itemize}
		\item[(i)] For all $f_1,\ldots,f_n \in \mathcal{A}$, it holds that
		\[
		\delta(I(f_1,\ldots,f_n))=f_1\cdot I(f_2,\ldots,f_n).
		\]
		\item[(ii)] For all integers $m,n\geq 0$, and all $f_1,\ldots,f_{m+n} \in \mathcal{A}$, we have
		\begin{equation} \label{eqn:shuffleproductIII}
			I(f_1,\ldots,f_m)I(f_{m+1},\ldots,f_{m+n})=\sum_{\sigma\in \Sigma_{m,n}}I(f_{\sigma^{-1}(1)},\ldots,f_{\sigma^{-1}(m+n)}).
		\end{equation}
	\end{itemize}
\end{proposition}
\begin{proof}
	Since $\pi: \mathcal{A}_+\rightarrow \mathcal{A}$ is an algebra homomorphism that commutes with $\delta$, this follows directly from Proposition \ref{prop:iteratedprimitivepropsII}.
\end{proof}
For later use, we record the following 'integration by parts' formula which follows immediately from Proposition \ref{prop:iteratedprimitivepropsIII}.(i) by induction on $n$.
\begin{corollary} \label{cor:partialintegration}
	For all $f_1,\ldots,f_n \in \mathcal{A}$, there exist constants $c_1,\ldots,c_{n-1} \in \ker(\delta)=\mathbb C$ such that
	\[
	\begin{aligned}
		I(\delta(f_1),f_2,\ldots,f_n)&=f_1\cdot I(f_2,\ldots,f_n)-I(f_1\cdot f_2,\ldots,f_n)+c_1,\\
		I(f_1,\ldots,f_{n-1},\delta(f_n))&=I(f_1,\ldots,f_{n-1}\cdot f_n)+\sum_{j=1}^{n-1}c_j I(f_1,\ldots,f_j),
	\end{aligned}
	\]
	and for every $1<i<n$, we have
	\[
	\begin{aligned}
		I(f_1,\ldots,\delta(f_i),\ldots,f_n)=I(f_1,\ldots,f_{i-1}\cdot f_i,\ldots,f_n)&-I(f_1,\ldots,f_i\cdot f_{i+1},\ldots,f_n)\\
		&+\sum_{j=1}^{i-1}c_j I(f_1,\ldots,f_{j-1}).
	\end{aligned}
	\]
\end{corollary}

\begin{remark}
	The constants $c_i$ in the preceding corollary are not necessarily equal to zero; see equation \eqref{eqn:integrationbypartsexample} above for an example where they are not.
\end{remark}
For later use, we record the following linear independence criterion for renormalized iterated primitives, which is a special case of \cite[Theorem 2.1]{DDMS}.
\begin{theorem}[Deneufch\^atel--Duchamp--Minh--Solomon] \label{thm:DDMS}
	Let $K\subset \mathcal{A}$ be a differential subfield and let $\mathcal{F}=\{f_i\}_{i\in I} \subset K$ be a family of elements of $K$. Then the following conditions are equivalent.
	\begin{itemize}
		\item[(i)]
		The family
		\[
		\bigcup_{m\geq 0}\{I(f_{i_1},\ldots,f_{i_m}) : (i_1,\ldots,i_m) \in \{1,\ldots,n\}^m\} \subset \mathcal{A}
		\]
		is $K$-linearly independent.
		\item[(ii)]
		The family $\mathcal{F}$ is $\mathbb C$-linearly independent and we have $d(K)\cap \operatorname{Span}_\mathbb C\mathcal{F}=\{0\}$
	\end{itemize}
	\qed
\end{theorem}

\section{The algebraic structure of iterated primitives of meromorphic quasimodular forms} \label{sec:5}

Consider the subfield $K:=\mathbb C(E_2,E_4,E_6,q)\subset \mathcal{A}$. It is closed under $\delta$.
\begin{definition}
	Define
	\[
	\mathcal{I}^{\mathcal{Q}\mathcal{M}}:=\operatorname{Span}_K\{ I(f_1,\ldots,f_n) \, : \, f_1,\ldots,f_n \in \mathcal{Q}\mathcal{M}, \, n\geq 0 \} \subset \mathcal{A}.
	\]
\end{definition}
\begin{remark}
	It follows from Proposition \ref{prop:iteratedprimitivepropsIII} that $\mathcal{I}^{\mathcal{Q}\mathcal{M}}$ is a differential $K$-subalgebra of $\mathcal{A}$.
\end{remark}
On the other hand, let $K\langle \mathcal{Q}\mathcal{M}\rangle$ denote the $K$-shuffle algebra on the $\mathbb C$-vector space $\mathcal{Q}\mathcal{M}$ (see for example \cite[Section 1.4]{Reutenauer} for generalities on shuffle algebras). Elements of $K\langle \mathcal{Q}\mathcal{M}\rangle$ are $K$-linear combinations of words $[f_1|\ldots|f_n]$, with $f_1,\ldots,f_n\in \mathcal{Q}\mathcal{M}$, and its product is given by the shuffle product, defined on words by
\[
[f_1|\ldots|f_m]\shuffle [f_{m+1}|\ldots|f_{m+n}]=\sum_{\sigma\in\Sigma_{m,n}}[f_{\sigma^{-1}(1)}|\ldots|f_{\sigma^{-1}(m+n)}],
\]
with $\Sigma_{m,n}$ as in Proposition \ref{prop:iteratedprimitiveprops}.(ii). By definition of $\mathcal{I}^{\mathcal{Q}\mathcal{M}}$, there is a surjective $K$-linear map
\begin{equation} \label{eqn:algebrahomomorphism}
	I: K\langle \mathcal{Q}\mathcal{M}\rangle\rightarrow \mathcal{I}^{\mathcal{Q}\mathcal{M}}, \qquad [f_1|\ldots|f_n] \mapsto I(f_1,\ldots,f_n),
\end{equation}
which is an algebra homomorphism, by Proposition \ref{prop:iteratedprimitivepropsIII}.(ii).

Now let $\mathcal{C}\subset \mathcal{Q}\mathcal{M}$ be a $\mathbb C$-linear complement of $\delta(\mathcal{Q}\mathcal{M})\subset \mathcal{Q}\mathcal{M}$, that is, $\mathcal{C}$ is chosen so as to satisfy $\mathcal{C}\oplus \delta(\mathcal{Q}\mathcal{M})=\mathcal{Q}\mathcal{M}$. We obtain a $K$-subalgebra $K\langle \mathcal{C}\rangle\subset K\langle \mathcal{Q}\mathcal{M}\rangle$.
\begin{theorem} \label{thm:main}
	The restriction of \eqref{eqn:algebrahomomorphism} to $K\langle \mathcal{C}\rangle$ induces an isomorphism
	$
	K\langle \mathcal{C}\rangle\cong \mathcal{I}^{\mathcal{Q}\mathcal{M}}.
	$
\end{theorem}
\begin{proof}
	We first prove surjectivity. Let $I(f_1,\ldots,f_n)\in \mathcal{I}^{\mathcal{Q}\mathcal{M}}$. We employ induction on $n$ to show that $I(f_1,\ldots,f_n) \in I(K\langle \mathcal{C}\rangle)$. By definition of $\mathcal{C}$, for every $1\leq i\leq n$, there exist $g_i \in \mathcal{C}$ and $h_i \in \mathcal{Q}\mathcal{M}$ such that $f_i=g_i+\delta(h_i)$. Using multilinearity of $I$, we deduce that
	\[
	I(f_1,\ldots,f_n)=I(g_1,\ldots,g_n)+\sum_{T\subsetneq \{1,\ldots,n\}}I(f^T_1,\ldots,f^T_n), \qquad \mbox{where } f_i^T:=
	\begin{cases}
		g_i, & \mbox{if }i\in T,\\
		\delta(h_i), & \mbox{if }i\notin T.
	\end{cases}
	\]
	Corollary \ref{cor:partialintegration} now implies that each term in the sum on the right hand side can be rewritten as a $K$-linear combination of iterated primitives of length $\leq n-1$, hence we obtain $I(f_1,\ldots,f_n)\in I(K\langle \mathcal{C}\rangle)$, as desired.
	
	For injectivity, by Theorem \ref{thm:DDMS} it suffices to verify that $\delta(K)\cap \mathcal{C}=\{0\}$, which in turn follows from the next lemma.
	\begin{lemma} \label{lem:main}
		We have $\delta(K)\cap \mathcal{Q}\mathcal{M}=\delta(\mathcal{Q}\mathcal{M})$.
	\end{lemma}
	\begin{proof}
		Let $K':=\operatorname{Frac}(\mathcal{Q}\mathcal{M})=\mathbb C(E_2,E_4,E_6) \subset K$ which is a differential subfield of $K$. We begin by showing that $\delta(K)\cap K'=\delta(K')$, the inclusion "$\supseteq$" being trivial. Let $f,g\in K[q]$ be such that $\delta(f/g) \in K'$. We may assume without loss of generality that $f$ and $g$ have no common factors and that $g=q^n+\sum_{j=0}^{n-1}g_j q^j$ is monic. We claim that $g=q^n$. Indeed, computing $\delta(f/g)$ by the quotient rule and using that $f$ and $g$ have no common factor, it must be the case that $g$ divides $\delta(g)$. Since $\delta(g)$ and $g$ have the same degree $n$ and since $g$ is monic, this implies that $\delta(g)=n\cdot g$, hence $g=q^n$. In particular, $f/g \in K'[q,q^{-1}]$ and since $\delta(q^n)=nq^n$, the fact that $\delta(f/g) \in K'$ readily implies $f/g\in K'$, as desired.
		
		We next show that $\delta(K')\cap \mathcal{Q}\mathcal{M}=\delta(\mathcal{Q}\mathcal{M})$, the inclusion "$\supseteq$" again being trivial. Let $f,g \in \mathbb C[E_2,E_4,E_6]$ be polynomials such that $\delta(f/g) \in \mathcal{Q}\mathcal{M}$, that is, we have $\delta(f/g)=h_1/h_2$ for some $h_1\in \mathbb C[E_2,E_4,E_6]$ and some $h_2\in \mathbb C[E_4,E_6]$, where $h_2$ is homogeneous for the weight. It is enough to prove that $g \in \mathcal{M}_k$ for some $k\in \mathbb Z$, and it is no loss of generality to assume that $f$ and $g$ have no common factors. Applying the quotient rule to compute $\delta(f/g)$, multiplying through by $h_2$ and reorganizing terms, we deduce that
		\[
		f\cdot \frac{\delta(g)}{g}\cdot h_2 \in \mathbb C[E_2,E_4,E_6],
		\]
		hence that $g$ divides $\delta(g)h_2$, as $f$ and $g$ have no common factor. We shall prove that this implies that $g \in \mathcal{M}_k$ for some $k$. Indeed, by the Leibniz rule each irreducible factor $g'$ of $g$ divides $\delta(g')h_2$, so it is enough to prove the claim for $g$ irreducible. If $g$ divides $h_2$, then we are done, since $h_2\in \mathbb C[E_4,E_6]$ is homogeneous by assumption, hence so is each of its factors. If $g$ does not divide $h_2$, then $g$ must divide $\delta(g)$ in $\mathbb C[E_2,E_4,E_6]$ in which case it is well-known that $g\in \mathbb C\cdot \Delta^m$ for some integer $m\geq 0$ (see for example \cite[Lemma 2.5]{Matthes:AlgebraicStructure}). Hence we obtain that $f/g\in \mathcal{Q}\mathcal{M}$, as desired.
	\end{proof}
	This ends the proof of Theorem \ref{thm:main}.
\end{proof}
We next give a basis of $\mathcal{I}^{\mathcal{Q}\mathcal{M}}$ in terms of an ordered $\mathbb C$-basis $B=(b_i)_{i\in I}$ of $\mathcal{C}$. Consider the free monoid $B^*$, that is, the set of all words in $B$ including the empty word $\emptyset$. A \emph{Lyndon word} is a nonempty word $w\in B^*$ such that for all factorizations $w=uv$ into nonempty words $u,v$, it holds that $w<v$, where "$<$" denotes the lexicographic order on $B^*$ induced from the order on $B$. We let $L(B)\subset B^*$ denote the subset of all Lyndon words. The following result is a special case of Radford's theorem, \cite[Theorem 3.1.1]{Radford}.
\begin{theorem}[Radford] \label{thm:Radford}
	The set 
	\begin{equation} \label{eqn:polynomialbasis}
		\{[b_{i_1}|\ldots|b_{i_n}] \, : \, w=b_{i_1}\ldots b_{i_n}\in L(B)\}
	\end{equation}
	is a polynomial basis of $K\langle \mathcal{C}\rangle$: every element of $K\langle \mathcal{C}\rangle$ can be written uniquely as a $K$-linear combination of shuffle products of elements of \eqref{eqn:polynomialbasis}.
\end{theorem}
Combined with Theorem \ref{thm:main}, this gives an explicit algebra basis of $\mathcal{I}^{\mathcal{Q}\mathcal{M}}$. As a consequence, we obtain the following algebraic independence criterion. We first need some more notation. For $f=\sum_{n>\!\!\!>-\infty}a_nq^n\in \mathcal{Q}\mathcal{M}$ and $r\geq 1$, we let
\[
I^r(f):=I(\underbrace{1,\ldots,1}_{r-1},f)
\]
denote the $r$-fold primitive of $f$. Note that
\[
I^r(f)=\frac{a_0t^r}{r!}+\sum_{n\neq 0}\frac{a_n}{n^r}q^n.
\]
\begin{theorem} \label{thm:criterion}
	Given $f_0,f_1,\ldots,f_n \in \mathcal{Q}\mathcal{M}$ with $f_0=1$, the subset
	\[
	\{t\}\cup \bigcup_{r\geq 1}\{I^r(f_i) \, : \, i\in \{1,\ldots,n\}\} \subset \mathcal{I}^{\mathcal{Q}\mathcal{M}}
	\]
	is algebraically independent over $K$, if and only if the equivalence classes of the $f_0,\ldots,f_n$ in $\mathcal{Q}\mathcal{M}/\delta(\mathcal{Q}\mathcal{M})$ are $\mathbb C$-linearly independent.
\end{theorem}
\begin{proof}
	If the classes of the $f_0,\ldots,f_n$ in $\mathcal{Q}\mathcal{M}/\delta(\mathcal{Q}\mathcal{M})$ are $\mathbb C$-linearly dependent, then there exist $\alpha_0,\ldots,\alpha_n \in \mathbb C$, not all equal to zero, and $g\in \mathcal{Q}\mathcal{M}$ such that $\sum_{i=0}^n\alpha_if_i=\delta(g)$. Applying $I$ to both sides then gives a nontrivial relation. Conversely, if $f_0,\ldots,f_n$ are $\mathbb C$-linearly independent modulo $\delta(\mathcal{Q}\mathcal{M})$, then $\operatorname{Span}_\mathbb C\{f_0,\ldots,f_n\} \cap \delta(\mathcal{Q}\mathcal{M})=\{0\}$. Hence there exists a complement $\mathcal{C}$ that contains $\operatorname{Span}_\mathbb C\{f_0,\ldots,f_n\}$, and by Theorem \ref{thm:main} it is enough to show that the subset
	\[
	\{[f_0]\} \cup \bigcup_{r\geq 1}\{[\underbrace{f_0|\ldots|f_0|}_{r-1}f_i] \, : \, i\in \{1,\ldots,n\}\} \subset K\langle \mathcal{C}\rangle
	\]
	is $K$-linearly independent. Since this set consists of Lyndon words (for the order $f_i<f_j$ if and only if $i<j$), the $K$-linear independence follows from Theorem \ref{thm:Radford}.
\end{proof}

\section{An explicit complement} \label{sec:6}

The remainder of this paper is devoted to constructing an explicit complement of $\delta(\mathcal{Q}\mathcal{M})$.

\subsection{Quotient by derivatives}
We begin by computing the quotient $\overline{\mathcal{Q}\mathcal{M}}:=\mathcal{Q}\mathcal{M}/\delta(\mathcal{Q}\mathcal{M})$. Since $\delta$ is homogeneous of weight $2$, the weight grading $\mathcal{Q}\mathcal{M}=\bigoplus_{k\in \mathbb Z}\mathcal{Q}\mathcal{M}_k$ descends to a grading $\overline{\mathcal{Q}\mathcal{M}}=\bigoplus \overline{\mathcal{Q}\mathcal{M}}_k$, where $\overline{\mathcal{Q}\mathcal{M}}_k:=\mathcal{Q}\mathcal{M}_k/\delta(\mathcal{Q}\mathcal{M}_{k-2})$. Let $\pi_k: \mathcal{M}_k\rightarrow \overline{\mathcal{Q}\mathcal{M}}_k$
be the composition of the natural inclusion $\mathcal{M}_k \subset \mathcal{Q}\mathcal{M}_k$ with the canonical projection onto $\overline{\mathcal{Q}\mathcal{M}}_k$. Note that $\ker(\pi_k)=\delta(\mathcal{Q}\mathcal{M}_{k-2})\cap \mathcal{M}_k$.
\begin{theorem} \label{thm:cokerdelta}
	The following statements hold.
	\begin{itemize}
		\item[(i)]
		We have
		\[
		\ker(\pi_k)\cong
		\begin{cases} 0, & \mbox{if }k\leq 1,\\
			\delta^{k-1}(\mathcal{M}_{2-k}), & \mbox{if }k\geq 2.
		\end{cases}
		\]
		\item[(ii)]
		We have
		\[
		\operatorname{coker}(\pi_k)\cong
		\begin{cases} 0, & \mbox{if }k\leq 1,\\
			\mathcal{M}_{2-k}\cdot E_2^{k-1}, & \mbox{if }k\geq 2.
		\end{cases}
		\]
	\end{itemize}
	In particular, we have
	\[
	\overline{\mathcal{Q}\mathcal{M}}_k\cong\left\{
	\begin{array}{cc}
		\displaystyle \mathcal{M}_k, & k\leq 1,\\
		\displaystyle
		\mathcal{B}_k\oplus \mathcal{M}_{2-k}\cdot E_2^{k-1}, & k\geq 2,\\
	\end{array}\right.
	\]
	where $\mathcal{B}_k:=\mathcal{M}_k/\delta^{k-1}(\mathcal{M}_{2-k})$.
\end{theorem}
\begin{proof}
	For (i), the inclusion "$\supseteq$" is a direct consequence of Bol's identity, \cite{Bol}, so that it remains to prove the converse inclusion. Let $f\in \delta(\mathcal{Q}\mathcal{M}_{k-2})\cap \mathcal{M}_k$ and $g\in \mathcal{Q}\mathcal{M}_{k-2}$ be such that $f=\delta(g)$, and let $g_0,\ldots,g_p$ denote the coefficient functions of $g$. Using equation \eqref{eqn:differentiation}, we compute that
	\begin{equation} \label{eqn:recursion}	f=\sum_{r=0}^{p+1}\left(\delta(g_r)+\frac{k-2-r+1}{12}g_{r-1}\right)X_\gamma^r.
	\end{equation}
	The leading coefficient, $\frac{k-2-p}{12}g_p$, must vanish by uniqueness of the coefficient functions, hence $p=k-2$ or $g_p=0$. If $k\leq 1$, the first case cannot happen and we have $g_p=0$. By recursion, we deduce that $g_r=0$ for all $r$, hence that $g=0$. If $k\geq 2$, we may assume without loss of generality that $g_p\neq 0$, hence that $p=k-2$ by uniqueness. By recursion on $r$, we deduce from \eqref{eqn:recursion} that $\delta(g_r)=-\frac{k-r-1}{12}g_{r-1}$, for all $1\leq r\leq p$. The statement follows from this since $g=g_0$ and $g_p\in \mathcal{M}_{2-k}$.
	
	We next prove (ii). Let $f\in \mathcal{Q}\mathcal{M}_k^{\leq p}$ with coefficient functions $f_0,\ldots,f_p$. By \eqref{eqn:polynomialinE2}, we can write $f=\sum_{r=0}^p \overline{f}_r\cdot E_2^r$ for uniquely determined $\overline{f}_r \in \mathcal{M}_{k-2r}$, and we have $\overline{f}_p=f_p$. Now define $h=\delta(\overline{f}_p\cdot E_2^{p-1}) \in \mathcal{Q}\mathcal{M}^{\leq p}_k$ and denote its coefficient functions by $h_0,\ldots,h_p$. Using equations \eqref{eqn:differentiation} and \eqref{eqn:E2quasimodular}, we compute that $h_p=\frac{k-p-1}{12}\overline{f}_p$, hence if $p\neq k-1$, we see that $\overline{f}_p\cdot E_2^p-\frac{12}{k-p-1}h \in \mathcal{Q}\mathcal{M}^{\leq p-1}_k$. By recursion on the depth, we deduce that the natural map $\mathcal{Q}\mathcal{M}_k^{\leq q}\rightarrow \overline{\mathcal{Q}\mathcal{M}}_k$ is surjective for $q>k-1$, and that the images of $\mathcal{Q}\mathcal{M}_k^{\leq k-2}$ and $\mathcal{M}_k$ inside of $\overline{\mathcal{Q}\mathcal{M}}_k$ are both equal. It remains to compute the image of $\mathcal{Q}\mathcal{M}_k^{\leq p}$ in $\overline{\mathcal{Q}\mathcal{M}}$ in the case $p=k-1$. Let $\widetilde{h}\in \delta(\mathcal{Q}\mathcal{M}_{k-2})$ be such that $\overline{f}_p\cdot E_2^{k-1}-\widetilde{h} \in \mathcal{Q}\mathcal{M}_k^{\leq k-2}$. This necessarily implies $\widetilde{h}_q=0$, for all $q>k-1$, and using equation \eqref{eqn:differentiation} in addition, we deduce that $\widetilde{h}_{k-1}=0$. Therefore $\overline{f}_p\cdot E_2-\widetilde{h}$ has depth $\leq k-2$ if and only if $\overline{f}_p=0$, ending the proof.	
	
\end{proof}

\subsection{A splitting}

We next compute a splitting of the projection $\mathcal{M}_k\rightarrow \mathcal{B}_k$ where $\mathcal{B}_k$ was introduced in Theorem \ref{thm:cokerdelta}.
\begin{definition}
	For an integer $k\geq 2$, define $\widetilde{\mathcal{M}}_k\subset \mathcal{M}_k$ to be the subspace of those meromorphic modular forms $f\in \mathcal{M}_k$ which satisfy $v_\infty(f)\geq -\dim S_k$, and $v_P(f)\geq 1-k$ for all $P\in X\setminus \{\infty\}$.
\end{definition}
\begin{theorem} \label{thm:splitting}
	The natural projection $\varphi_k: \widetilde{\mathcal{M}}_k\rightarrow \mathcal{B}_k$ is an isomorphism.
\end{theorem}
\begin{proof}
	We first prove that $\varphi_k$ is injective. Let $f\in \ker(\varphi_k)=\widetilde{\mathcal{M}}_k\cap \delta^{k-1}(\mathcal{M}_{2-k})$ and write $f=\delta^{k-1}(g)$ for some $g\in \mathcal{M}_{2-k}$. Since $v_P(f)\geq 1-k$, an elementary argument shows that $v_P(g)\geq 0$, for every $P\in X\setminus \{\infty\}$. We next show that $v_\infty(g)\geq 0$. Assume for a contradiction that $v_\infty(g)<0$. Then $v_\infty(g)=v_{\infty}(\delta^{k-1}(g))$ and the valence formula yields that
	\begin{equation} \label{eqn:inequalityvalence2}
		\sum_{P\in X\setminus \{\infty\}}\frac{v_P(g)}{h_P}=\frac{2-k}{12}-v_\infty(g) \leq \frac{2-k}{12}+\dim S_k.
	\end{equation}
	On the other hand, using the well-known formula for $\dim S_k$, equation \eqref{eqn:inequalityvalence2} leads to a contradiction. Therefore $v_\infty(g)\geq 0$, hence $g$ is a holomorphic modular form of weight $2-k\leq 0$, which must be constant. Therefore, $f=0$ and $\varphi_k$ is injective.
	
	For surjectivity, assume that we are given $f\in \mathcal{M}_k$. If there exists $P\in X\setminus \{\infty\}$ such that $v_P(f)\leq -k$, then by Lemma \ref{lem:construction}.(i), there exists $g\in \mathcal{M}_{2-k}$ with $v_P(g)=v_P(f)+k-1$ and $v_Q(g)\geq 0$ for all $Q\in X\setminus \{P,\infty\}$. Since $v_P(\delta^{k-1}(g))=v_P(f)$, there exists $\alpha \in \mathbb C^\times$ such that $v_P(f-\alpha\delta^{k-1}(g))\geq v_P(f)+1$ and $v_Q(f-\alpha\delta^{k-1}(g))\geq \min\{0,v_Q(f)\}$, for all $Q\in X\setminus \{P,\infty\}$. Hence, by iterating this procedure, we find $f' \in \mathcal{M}_k$ such that $f-f' \in \delta^{k-1}(\mathcal{M}_{2-k})$ and such that $v_P(f')\geq 1-k$ for all $P\in X\setminus \{\infty\}$.
	
	Similarly, if $v_\infty(f')\leq -\dim S_k-1$, then by Lemma \ref{lem:construction}.(ii) there exists $g' \in \mathcal{M}_{2-k}$ such that $v_\infty(g')=v_\infty(f')$ and $v_P(g')\geq 0$, for all $P\in X\setminus \{\infty\}$. Since $v_\infty(\delta^{k-1}(g'))=v_\infty(f')$, there exists $\alpha' \in \mathbb C^\times$ such that $v_\infty(f'-\alpha'\delta^{k-1}(g'))=v_\infty(f')+1$ and $v_P(f'-\alpha'\delta^{k-1}(g'))\geq \min\{0,v_P(f')\}$, for all $P\in X\setminus \{\infty\}$. Iterating this procedure, we find $f''\in \widetilde{\mathcal{M}}_k$ such that $f'-f'' \in \delta^{k-1}(\mathcal{M}_{2-k})$. In particular, $f-f''\in \delta^{k-1}(\mathcal{M}_{2-k})$, from which it follows that $\varphi_k$ is surjective.
\end{proof}
Combining Theorems \ref{thm:cokerdelta} and \ref{thm:splitting}, we can finally give an explicit complement of the subspace of derivatives.
\begin{theorem}\label{thm:complement}
	\[
	\mathcal{Q}\mathcal{M}_k=
	\begin{cases}
		\delta(\mathcal{Q}\mathcal{M}_{k-2})\oplus \mathcal{M}_k, &\mbox{if } k\leq 1,\\
		\delta(\mathcal{Q}\mathcal{M}_{k-2})\oplus \mathcal{M}_{2-k}\cdot E_2^{k-1} \oplus \widetilde{\mathcal{M}}_k, & \mbox{if } k\geq 2.
	\end{cases}
	\]
\end{theorem}

\subsection{Dimension formulas}

Using Theorem \ref{thm:splitting}, we can also compute dimension formulas for the space $\mathcal{B}_k$ after restricting the location of poles. 
Given a finite subset $S\subset X$, we define
\[
\begin{aligned}
	M_k(\ast S)&=\{f \in \mathcal{M}_k \, : \, v_P(f)\geq 0, \, \forall P \in X\setminus S\},\\
	\widetilde{M}_k(\ast S)&=M_k(\ast S)\cap \widetilde{\mathcal{M}}_k.
\end{aligned}
\]

\begin{theorem}\label{thm:dimension}
	Let $S\subset X$ be a finite subset with $\infty \in S$, and let $k\geq 2$ be an even integer. Then
	\[
	\dim \widetilde{M}_k(\ast S)=\dim M_k+\dim S_k+\sum_{P\in S'}w_P(k),
	\]
	where $S':=S\setminus \{\infty\}$ and
	\[
	w_P(k):=\begin{cases}
		2\left\lfloor \frac{k-2}{4}\right\rfloor+1, & \mbox{if }P=[i],\\
		2\left\lfloor \frac{k-2}{6}\right\rfloor+1, &  \mbox{if }P=[e^{2\pi i/3}],\\
		k-1, &\mbox{else.}
	\end{cases}
	\]
\end{theorem}
\begin{proof}
	Consider the divisor
	$
	D_S=\dim S_k\cdot (\infty)+\sum_{P\in S'}\frac{k-1}{h_P}\cdot (P).
	$
	Then $M_k(D_S)=\widetilde{M}_k(\ast S)$, by definition, and $\dim \widetilde{M}_k(\ast S)=\dim M_{k+12\deg D_S}$ by equation \eqref{eqn:dimension}. We now distinguish four cases.
	\bigskip
	
	\noindent
	\textit{Case 1: $[i],[e^{2\pi i/3}] \notin S$.}
	In this case, $\deg D_S=\dim S_k+|S'|\cdot (k-1)$ and $\dim M_{k+12\deg D_S}=\dim M_k+\deg D_S=\dim M_k+\dim S_k+\sum_{P\in S'}w_P(k)$, whence the result.
	\bigskip
	
	\noindent
	\textit{Case 2: $[i] \in S$, $[e^{2\pi i/3}] \notin S$.}
	In this case, similarly to before, we deduce
	\[
	\dim \widetilde{M}_k(\ast S)=\dim M_{7k-6}+\dim S_k+\sum_{P \in S' \setminus \{[i]\}}w_P(k).
	\]
	It remains to prove that $\dim M_{7k-6}=\dim M_k+w_{[i]}(k)$, for all even $k\geq 2$. In fact, a direct computation shows that the term $\dim M_{7k-6}-\dim M_k-w_{[i]}(k)$ is invariant under $k\mapsto k+12$, so that it is enough to prove the statement for $2\leq k\leq 12$, where it can be verified directly.
	\bigskip
	
	\noindent
	\textit{Case 3: $[i] \notin S$, $[e^{2\pi i/3}] \in S$.}
	This is analogous to the previous case, except that we now have to verify that $\dim M_{5k-4}=\dim M_k+w_{[e^{2\pi i/3}]}(k)$, which can be proved similarly to before.
	\bigskip
	
	\noindent
	\textit{Case 4: $[i],[e^{2\pi i/3}] \in S$.}
	Similarly, it now suffices to verify that $\dim M_{11k-10}=\dim M_k+w_{[i]}(k)+w_{[e^{2\pi i/3}]}(k)$. This is proved similarly to before.

\end{proof}
\begin{remark}
	For a finite subset $S\subset X$ with $\infty \in S$, let $B_k(\ast S)\subset \mathcal{B}_k$ be the image of the natural injection
	\[
	M_k(\ast S)/\delta^{k-1}(M_{2-k}(\ast S))\hookrightarrow \mathcal{B}_k.
	\]
	The isomorphism $\varphi_k: \widetilde{\mathcal{M}}_k\rightarrow \mathcal{B}_k$ of Theorem \ref{thm:splitting} then restricts to an isomorphism $\widetilde{M}_k(\ast S)\cong B_k(\ast S)$, hence Theorem \ref{thm:dimension} implies a dimension formula for $B_k(\ast S)$. In the case $S=\{\infty\}$ (the "weakly holomorphic case"), this is due to Guerzhoy, \cite{Guerzhoy}.
\end{remark}

\subsection{Algebraic independence -- revisited}
As a final corollary, we prove a generalization of \cite[Theorem 5]{PasolZudilin}. With notation as in \textit{loc. cit.}, consider the meromorphic modular forms
\[
F_{4a}(\tau)=\frac{\Delta(\tau)}{E_4(\tau)^2}, \qquad F_{4b}(\tau)=\frac{E_4(\tau)\Delta(\tau)}{E_6(\tau)^2}, \qquad F_6(\tau)=\frac{E_6(\tau)\Delta(\tau)}{E_4(\tau)^3}.
\]
\begin{corollary} \label{cor:PasolZudilin}
	The set
	\[
	\{t\}\cup\bigcup_{r\geq 1}\{I^r(f) \, : \, f\in \{F_{4a},F_{4b},F_6\}\}
	\]
	is algebraically independent over $K$.
\end{corollary}
\begin{proof}
	By Theorem \ref{thm:criterion}, it is sufficient to prove that the classes of $1$, $F_{4a},F_{4b}$, and $F_6$ in $\mathcal{Q}\mathcal{M}/\delta(\mathcal{Q}\mathcal{M})$ are $\mathbb C$-linearly independent. Since $E_4$ and $E_6$ both have at most a simple zero at every $\tau\in \mathfrak{H}$ and no zero at $\infty$, one sees that $F_{4a},F_{4b} \in \widetilde{\mathcal{M}}_4$ and that $F_6\in \widetilde{\mathcal{M}}_6$. The $\mathbb C$-linear independence of their classes modulo $\delta(\mathcal{Q}\mathcal{M})$ now follows from Theorem \ref{thm:complement}.
\end{proof}

\bibliographystyle{amsplain}

\end{document}